\documentclass[opre,nonblindrev]{informs3} 

\OneAndAHalfSpacedXI 

\usepackage{endnotes}
\let\footnote=\endnote
\let\enotesize=\normalsize
\def\notesname{Endnotes}%
\def\makeenmark{$^{\theenmark}$}
\def\enoteformat{\rightskip0pt\leftskip0pt\parindent=1.75em
  \leavevmode\llap{\theenmark.\enskip}}

\usepackage{natbib}
 \bibpunct[, ]{(}{)}{,}{a}{}{,}%
 \def\bibfont{\small}%
 \def\bibsep{\smallskipamount}%
 \def\bibhang{24pt}%
 \def\newblock{\ }%

\usepackage{bbm}
\usepackage{amsfonts,amssymb}
\usepackage{enumerate}
\usepackage{amsmath}                              
\usepackage{CJK}
\usepackage{verbatim}
\usepackage{setspace}
\usepackage{changepage}
\usepackage{cite}
\usepackage{amssymb}
\usepackage{algorithm}  
\usepackage{algorithmicx}  
\usepackage{algpseudocode}  
\TheoremsNumberedThrough     
\ECRepeatTheorems

\EquationsNumberedThrough    

\begin{document}


\RUNAUTHOR{}

\RUNTITLE{}

\TITLE{Two-stage Stochastic Lot-sizing Problem with Chance-constrained Condition in the Second Stage}

\ARTICLEAUTHORS{%
\AUTHOR{Zeyang Zhang, Chuanhou Gao}
\AFF{School of Mathematical Sciences, Zhejiang University, Hangzhou, \EMAIL{zy\_zhang@zju.edu.cn, gaochou@zju.edu.cn}} 
\AUTHOR{Shabbir Ahmed}
\AFF{School of Industrial \& Systems Engineering, Georgia Institute of Technology, Atlanta, \EMAIL{shabbir.ahmed@isye.gatech.edu}}
} 

\ABSTRACT{%
In a given production planning horizon, the demands may only be comfirmed in part of the whole periods, and the others are uncertain. In this paper, we consider a two-stage stochastic lot-sizing problem with chance-constrained condition in the second stage. In the first stage, the demand is deterministic in each period, while in the second stage, the demands are random variables with finite distribution. We prove the optimality condition of the solutions under modified Wagner-Whitin condition and further develop a new equivalent MIP formulation which can depict the feasible region better. We also show that any subproblem fixing the selection of occurred scenarios has a complete linear description of its convex hull. A branch-and-bound algorithm is provided utilizing the character of the given subproblems. 
}%



\maketitle

%


\renewcommand{\algorithmicrequire}{\textbf{Initialization:}}
\renewcommand{\algorithmicensure}{\textbf{Main loop:}}
\section{Introduction}
Production planning problems are a common kind of problem that we usually meet in real life, and a popular type of that is lot-sizing problem. The deterministic uncapacitated lot-sizing problem(ULS) (without backlogging) is first proposed by \citet{1doi:10.1287/mnsc.5.1.89}, which is to determine a production plan for a product to satisfy demands over a finite time horizon while minimizing the summation of setup, production, and inventory holding costs. An $ \mathcal{O}(T^2) $ algorithm is proposed by the authors for ULS, where $ T $ is the total number of time periods in the planning horizon. \citet{4doi:10.1287/mnsc.37.8.909} and \citet{5doi:10.1287/opre.40.1.S145} improve the polynomial algorithm so that it can be solved in $ \mathcal{O}(T\log T) $ time and even in $ \mathcal{O}(T) $ time with Wagner-Whitin property. An explicit convex hull description of ULS is given by \citet{2Barany1984} utilizing the so-called $ (l, S) $ inequalities. 

The first polyhedral study of deterministic ULS problem with backlogging (ULSB) is performed by \citet{3Pochet1988}, in which the authors reformulate the structure of the problem introducing new variables to obtain extended formulations by several methods. The complete linear description of the convex hull of ULSB is provided by \citet{12Küçükyavuz2009} by generalizing the inequalities of \citet{3Pochet1988}. Similar $ \mathcal{O}(T\log T) $ algorithm for ULSB like that for ULS is proposed by \citet{7doi:10.1002/1520-6750(199306)40:4<459::AID-NAV3220400404>3.0.CO;2-8}. In addition, \citet{8Pochet1994} carry out the polyhedra study of lot-sizing problem under several different conditions with Wagner-Whitin costs.

When considering the indeterminacy of the demand in each period, the stochastic uncapacitated lot-sizing problem (SULS) is proposed. \citet{10doi:10.1287/opre.51.3.461.14960} and \citet{11Ahmed2003} study the stochastic capacity expansion problems which include the form SULS as a submodel. Furthermore, \citet{13Guan2006} study the polyhedral of SULS based on scenario tree. They provide several kinds of valid inequalities, and give a sufficient condition under which those inequalities are facet-defining. As for algorithm study, \citet{14doi:10.1287/opre.1070.0479} propose an $ \mathcal{O}(n^3\log{\mathcal{C}}) $ time dynamic algorithm for SULS, where $ n $ is the number of nodes in the scenario tree and $ \mathcal{C} $ is the maximum number of children for each node in the tree. Similar algorithms can be generalized to SULS with random lead times (\citet{15HUANG2008303} and \citet{16JIANG201174}). 

\citet{21Liu2018} consider that the stochastic lot-sizing model may lead to an over-conservative solution with excessive inventory, because the uncertain demand in each time period has to be satisfied. Thereby, a chance-constrained lot-sizing formulation is introduced, which is referred to as the static probabilistic lot-sizing problem (SPLS). SPLS assumes that the demands over the planning horizon are random with finite distribution, and for a given service level, $ 1-\varepsilon $, the production schedule only need to meet the demands with probability at least $ 1-\varepsilon $ at the beginning of the planning horizon. The first variant of SPLS is provided by \citet{19doi:10.1080/1055678021000033937}, however, the objective function does not consider the inventory cost. \citet{31Küçükyavuz2012}, \citet{32Abdi2016} and \citet{33Zhao2017} all solve the SPLS model with the inventory costs in branch-and-cut algorithms, which are performed as the testification of the validity of their proposed valid inequalities for general chance-constrained programming problems. Another dynamic variant of SPLS that updates the production schedule after the scenario realization of the former time periods is studied by \citet{20doi:10.1287/mnsc.2013.1822}. For the polyhedral study, \citet{21Liu2018} give the first relevant result that exploits the lot-sizing structure into the construction of valid inequalities and facet-defining inequalities for SPLS, however, the study is under a equiprobable condition.

Chance-constrained programming (CCP) is introduced by \citet{22doi:10.1287/mnsc.4.3.235} and \citet{23doi:10.1287/opre.11.1.18}, whose feasible solutions satisfy the constraints with probability over $ 1-\varepsilon $, where $ \varepsilon $ is a given threshold. This kind of problems has been extensively studied, details about background and a list of references can be seen in \citet{24PREKOPA2003267}. CCP problem with stochastic right-hand side under a finite discrete distribution is a significant class of CCP that is explored broadly, and the CCP in our paper belongs to that type as well. To deal with the deterministic equivalent formulation of the former kind of CCP problem, an collection of efficient valid inequalities called mixing inequalities is introduced by \citet{26Günlük2001} over mixing set, which refers to the method in the study of vertex packing problem (\citet{27Atamtürk2000}). Based on mixing set, \citet{30Luedtke2010} and \citet{31Küçükyavuz2012} give some stronger valid inequalities, and study under which conditions the proposed valid inequalities are sufficient to be facet-defining. \citet{32Abdi2016} explore the characterization of valid inequalities for single mixing set, and explicitly develop a set of facet-defining inequalities under some particular conditions. \citet{33Zhao2017} generalize the valid and facet-defining inequalities presented in \citet{31Küçükyavuz2012} and \citet{32Abdi2016}, expect that another family of valid inequalities called knapsack cover inequalities is provided by lifting techniques. 

\citet{18Zhou2013} propose a two-stage stochastic lot-sizng problem (with backlogging), in which the planning horizon is separated into two stages, in the first stage the cost parameters are deterministic, while in the second stage the cost parameters are random variables, and the demands in whole horizon are deterministic. The authors give a high-dimensional complete linear description of the convex hull of that problem. Now we consider a different kind of two-stage ULS (without backlogging): the demands can be confirmed in some latest	periods, however, beyond these periods, they become uncertain, meanwhile the cost parameters are deterministic in each period. Briefly, the demands are fixed in the first stage, and random variables in the second stage, which is likely to occur in the real life when a long-term production planning is going to be made. In addition, we assume the second stage random demands obey a finite distribution, and introduce a chance-constrained condition to avoid over-conservative solutions, like \citet{21Liu2018} do, but without the limitation of equiprobable condition. We expect to provide a polyhedral study of our proposed two-stage stochastic lot-sizing problem with chance-constrained condition in the second stage (SLSCC). 

The remaining part of this paper is organized as follows: In Sect.2, we depict the necessary notations and the mathmematical formulation of SLSCC. In Sect.3.1, we define the modified Wagner-Whitin condition for SLSCC, and discover the optimality condition of the solutions which can be used to generate a better new equivalent MIP formulation. In Sect.3.2, we prove that there is a complete linear description of the convex hull for any $ S $-subproblem which is obtained by fixing the occurred scenarios for the new formulation. In Sect.4, we provide a branch-and-bound algorithm utilizing the good character of $ S $-subproblem. 

\section{Mathematical Formulation}
In this paper, we consider a planning horizon with length $ T $, let $ N:=\{1,\ldots,T\} $. We assume that the demands for period $ 1 $ to $ p $ are deterministic, $ p\in N $, while for the remaining periods are uncertain and follow a discrete probability distribution with finite support, so that the entire planning horizon is separated into two stages. For convinence, we define $ [a,b]=\{a,a+1,\ldots,b-1,b\} $, for any integers $ a\le b $, throughout the rest of paper. The corresponding two-stage stochastic lot-sizing problem with chance-constrained condition in the second stage (SLSCC) then can be formulated as follows:
\begin{align}
\min \quad& {\alpha^1}^\top x^1+{\beta^1}^\top y^1+{h^1}^\top s^1+{\alpha^2}^\top x^2+{\beta^2}^\top y^2+{\mathbb{E}}_\xi ({\Theta}_\xi(x^2))\nonumber\\
&x_i+s_{i-1}=d_i+s_i, \quad i\in [1,p]\\
&\mathbb{P}(s_p+\sum\limits_{t=p+1}^{i} x_t\ge \sum\limits_{t=p+1}^{i} \xi_t, \ i\in[p+1,T])\ge 1-\varepsilon \\
& x_i\le M_i y_i, \quad i\in N\\
& x^1,\ s^1 \in\mathbb{R}_+^p,\quad x^2\in \mathbb{R}_+^{T-p}, \quad y^1\in\{0, 1\}^p, \quad y^2\in\{0, 1\}^{T-p}, 
\end{align}
where $ x^1= (x_1, \ldots, x_p)$, $y^1=(y_1, \ldots, y_p)$ and $s^1=(s_1, \ldots, s_p) $ represent the production level vector, set up decision vector and inventory level vector in the first stage respectively. $ \alpha^1= (\alpha_1, \ldots, \alpha_p) $, $ \beta^1=(\beta_1, \ldots, \beta_p) $ and $ h^1=(h_1, \ldots, h_p) $ are the unit production cost vector, fixed setup cost vector, and holding cost vector corresponding to $ x^1 $, $ y^1 $ and $ s^1 $. $ x^2=(x_{p+1}, \ldots, x_T) $, $ y^2=(y_{p+1}, \ldots, y_T) $ and $ \alpha^2=(\alpha_{p+1}, \ldots, \alpha_T) $, $ \beta^2=(\beta_{p+1}, \ldots, \beta_T) $ have the similar meaning in the second stage. $ \xi=(\xi_{p+1}, \cdots, \xi_T) $ is the uncertain demand vector in the second stage, and $ \varepsilon $ is a given threshold by which the probability of an undesirable outcome is limited. Constraints $ (1) $ are the relations among production, inventory and demand in the first stage. Constraint $ (2) $ ensures that the probability of violating the demands from period $ p+1 $ to $ T $ should be less than the given risk rate $ \epsilon $. $ M_i $ is a large constant to make constraints $ (3) $ redundant when $ y_i $ equals to one, for all $ i\in N $. In addition, $ {\Theta}_\xi(x^2) $ is the value function given by:
\begin{align}
{\Theta}_\xi(x^2)=& \min \quad {h^2}^\top s^2(\xi) \nonumber\\
&s_i(\xi)\ge [s_p+\sum\limits_{t=p+1}^{i} (x_t-\xi_t)] \mathbbm{1}_{\xi},\quad i\in[p+1,T]\\
&s^2(\xi)\in \mathbb{R}_+^{T-p},
\end{align}
where $ s^2(\xi)=(s_{p+1}(\xi), \cdots, s_T(\xi) )$ is the vector of second-stage inventory variables related to the realization of uncertain demand vector $ \xi $, and $ h^2 $ is the corresponding nonnegative holding cost vector. $ \mathbbm{1}_{\xi} $ is an indicator function, which equals to $ 1 $ when scenario is chosen, and $ 0 $ otherwise. Constraints $ (5) $ and $ (6) $ guarantee that the inventory level can be calculated correctly for the corresponding demand realization.

Assume the finite scenario set $ \Omega=\{1,\ldots,m\} $, let $ p_j $ be the probability of scenario $ j $, for all $ j\in \Omega $. In addition, let $ d_{ji} $ be the demand for period $ i $ under scenario $ j $, for all $ i\in N $ and $ j\in \Omega $. Let $ s_{ji} $ be the inventory at the end of time period $ i\in N $ in scenario $ j\in \Omega $, which incurs a unit holding cost $ h_i $. Then we can transform the formulation of SLSCC into a deterministic equivalent formulation as (refer to \citet{21Liu2018}):
\begin{align}
\min \quad &{\alpha^1}^\top x^1+{\beta^1}^\top y^1+{h^1}^\top s^1+{\alpha^2}^\top x^2+{\beta^2}^\top y^2+\sum\limits_{j=1}^{m} p_j {h^2}^\top s_j^2\nonumber\\
&x_i+s_{i-1}=d_i+s_i, \quad i\in[1,p]\\
&s_p+\sum\limits_{i=p+1}^{t} x_i\ge \sum\limits_{i=p+1}^{t} d_{ji}(1-z_j)_,\quad t\in[p+1,T],\quad j\in \Omega\\
&\sum\limits_{j=1}^{m} p_j z_j\le \varepsilon\\
&s_{ji} \ge [\sum\limits_{t=p+1}^{i} (x_t-d_{jt}) +s_p](1-z_j) ,\quad i\in[p+1,T],\quad j\in \Omega \\
&x_i\le M_i y_i \quad i\in N \\
&x^1,\ s^1 \in\mathbb{R}_+^p,\quad x^2,\ s_j^2\in \mathbb{R}_+^{T-p}, \quad y^1\in\{0, 1\}^p, \quad y^2\in\{0, 1\}^{T-p},\quad z\in\{0,1\}^m
\end{align}
where $ z_j $ is the introduced additional indicator variable, which equals to 0 if the demand in each time period under scenario $ j $ is satisfied, and 1 otherwise, for all $ j\in\Omega $. $ M_i=\sum\limits_{t=i}^{p}d_t +\max_{j\in\Omega}\{\sum\limits_{t=p+1}^{T}d_{jt}\} $, for $ i\in[1,p] $, and $ M_i=\max_{j\in\Omega}\{\sum\limits_{t=i}^{T}d_{jt}\} $, for $ i\in[p+1,T] $. Since the deterministice equivalent formulation can only yield a very weak linear programming relaxation, the polyhedral structure of that need further study then. In the next section, we will show there is a better equivalent formulation under a stronger Wagner-Whitin condition defined by us.

\begin{remark}
	Note that for constraint $ (5) $ we make a minor change comparing to that in the formlulation of \citet*{21Liu2018}, i.e., we multiply an indicator $ \mathbbm{1}_{\xi} $ on the right-hand side of the inequality. Through this handling, the second-stage inventory level of every period of unchosen scenarios will be zero and thus not produce cost to the objective function. Without multiplying the indicator, for any unchosen scenario $ \xi $, $ s_p+\sum\limits_{t=p+1}^{i} (x_t-\xi_t) $ can be negative for some period $ i\in[p+1,T] $, then the optimal inventory level $ s_i^2(\xi) $ will be zero for such kind of period $ i $ because of constraint $ (6) $. In fact, $ s_i^2(\xi)=s_p+\sum\limits_{t=p+1}^{i} (x_t-\xi_t) $ is just the real inventory level in period $ i\in[p+1,T] $ for demand realization $ \xi $, and when it is negative means that demand is not satisfied in this period, which incurs backlogging. In \citet*{21Liu2018}, the authors add the cost of real inventory level of unchosen scenario to objective function when it is nonnegative, and omit the cost of backlogging when it is negative. However, we consider that the cost of both inventory and backlogging for unchosen scenarios should not be included in the objective function, because we do not care about any influence of the unchosen scenarios during production. Though our handling will incur a kind of nonlinear constraint $ (10) $ in the deterministic equivalent formulation, we will show that under an assumption there is a mixed-integer linear  formulation have the same optimal solutions as original formulation in Sect.3.1. In addition, we will also show that a kind of subproblem has a good property in Sect.3.2.
\end{remark}

\section{Optimality Condition and New Formulation}
In this section, we first make a modified Wagner-Whitin costs assumption, and then based on that study the optimal solution forms of production and inventory for SLSCC. Furthermore, we generate a refoumulation by the optimal solution forms, which can depict a much better polyhedral structure of the feasible region. In addition, we define a kind of subproblem by restricting the occurred scenarios, and then show we can construct the convex hull of the feasible region of those subproblems. 

\subsection{An Equivalent MIP Formulation}
The stronger version of the Wagner-Whitin condition as follows:
\begin{assumption}
	(Stronger Wagner-Whitin condition)
	For the two-stage SLSCC problem, it satisfies the following conditions:
	\begin{center}
		$ \alpha_i+h_i \ge \alpha_{i+1}$, for $ i=1,\ldots,p$, \quad$ \alpha_i+(1-\varepsilon)h_i \ge \alpha_{i+1}$, for $ i=p+1,\ldots,T-1 $.
	\end{center}  
	
\end{assumption}

Wagner-Whitin condition is a classical assumption for lot-sizing problems, which means the sum of current period's unit production cost and unit inventory cost more than next period's unit production cost. Under that condition, at least one of the optimal solutions satisfies that there is no inventory when starting production. We want to maintain the property in our problems as well, namely, apart from the same condition in the first stage, there is an optimal solution satisfy that at least one of the occurred scenarios' inventory is exhausted when starting production in the second stage. Therefore, we strengthen the Wager-Whitin condition somewhat. Our assumption is valid in many practical problems because $ \varepsilon $ is usually small, then it's very possible to hold if traditional Wager-Whitin condition holds.

For any period $ i $, let $ \psi(i) $ be the time period of the earliest descendant of period $ i $ which is set up, i.e., $ \psi(i)=min\{j: y_j=1, j\in[i+1,T]\} $, $ \phi(i) $ be the time period of the lastest ascendant of period $ i $ which is set up, i.e., $ \phi(i)=max\{j: y_j=1, j\in[1,i-1]\} $, and $ J_z $ be the index set of occurred scenarios related to a certain indicator vector $ z $, i.e., $ J_z=\{j: z_j=0, j\in\Omega\} $. Then we can describe the property of the optimal solution in the following proposition.

\begin{proposition}
	For the two-stage SLSCC problem, under Assumption 1, there exists an optimal production level of the form:\\
	if $ x_i> 0 $, then 
	\begin{eqnarray} 
	x_i &=& 
	\left\{ \begin{array}{lll} 
	\sum\limits_{t=i}^{\psi(i)-1}d_t, \ \ \psi(i)\le p. \\
	\sum\limits_{t=i}^{p}d_t+\max\limits_{\tau\in{J_z}}\{\sum\limits_{t=p+1}^{\psi(i)-1} d_{\tau t}\}, \ \ i\le p, \ \ \psi(i)\ge p+1. \\
	\max\limits_{\tau\in{J_z}}\{\sum\limits_{t=i}^{\psi(i)-1}d_{\tau t}-s_{\tau(i-1)}\}, \ \ i\ge p+1.
	\end{array} 
	\right. 
	\end{eqnarray}
	and an optimal inventory level of the form:
	\begin{eqnarray} 
	s_i &=& 
	\left\{ \begin{array}{lll} 
	\sum\limits_{t=i+1}^{\psi(i)-1}d_t, \ \ \psi(i)\le p. \\
	\sum\limits_{t=i+1}^{p}d_t+\max\limits_{\tau \in{J_z}}\{\sum\limits_{t=p+1}^{\psi(i)-1} d_{\tau t}\}, \ \ i\le p, \ \ \psi(i)\ge p+1.
	\end{array} 
	\right. 
	\end{eqnarray}
	
	if $ z_j=0 $, then
	\begin{eqnarray} 
	s_{ji} &=& 
	\left\{ \begin{array}{lll} 
	\max\limits_{\tau\in{J_z}}\{\sum\limits_{t=p+1}^{\psi(i)-1} d_{\tau t}\}-\sum\limits_{t=p+1}^{i}d_{jt}, \ \ \phi(i)\le p, \ \ i\ge p+1. \\
	x_i+s_{j(i-1)}-d_{ji}, \ \ y_i=1, \ \ i\ge p+1. \\
	x_{\phi(i)}+s_{j(\phi(i)-1)}-\sum\limits_{t=\phi(i)}^{i}d_{jt}, \ \ y_i=0, \ \ \phi(i)\ge p+1.
	\end{array} 
	\right. 
	\end{eqnarray}
	
	if $ z_j=1 $, then $ s_{ji}=0 $ for $ i\in[p+1,T] $.
\end{proposition}

Since the production level $ x_i $ and the inventory level in the second stage $ s_{ji} $ are not only expressed by demands, it is not a good formula to reconstruct the feasible region. Fortunately, we can prove the following proposition.

\begin{proposition}
	For the two-stage SLSCC problem, under Assumption 1, there exists an optimal production level of the form:\\
	if $ x_i> 0 $, then 
	\begin{eqnarray} 
	x_i &=& 
	\left\{ \begin{array}{lll} 
	\sum\limits_{t=i}^{\psi(i)-1}d_t, \ \ \psi(i)\le p. \\
	\sum\limits_{t=i}^{p}d_t+\max\limits_{\tau\in{J_z}}\{\sum\limits_{t=p+1}^{\psi(i)-1} d_{\tau t}\}, \ \ i\le p, \ \ \psi(i)\ge p+1. \\
	\max\limits_{\tau\in{J_z}}\{\sum\limits_{t=p+1}^{\psi(i)-1}d_{\tau t}\}-\max\limits_{\tau\in{J_z}}\{\sum\limits_{t=p+1}^{i-1} d_{\tau t}\}, \ \ i\ge p+1.
	\end{array} 
	\right. 
	\end{eqnarray}
	and an optimal inventory level of the form:
	\begin{eqnarray} 
	s_i &=& 
	\left\{ \begin{array}{lll} 
	\sum\limits_{t=i+1}^{\psi(i)-1}d_t, \ \ \psi(i)\le p. \\
	\sum\limits_{t=i+1}^{p}d_t+\max\limits_{\tau \in{J_z}}\{\sum\limits_{t=p+1}^{\psi(i)-1} d_{\tau t}\}, \ \ i\le p, \ \ \psi(i)\ge p+1.
	\end{array} 
	\right. 
	\end{eqnarray}
	
	if $ z_j=0 $, then
	\begin{eqnarray} 
	s_{ji} &=&  
	\max\limits_{\tau\in{J_z}}\{\sum\limits_{t=p+1}^{\psi(i)-1} d_{\tau t}\}-\sum\limits_{t=p+1}^{i}d_{jt}, \quad i\in[p+1,T].
	\end{eqnarray}
	
	if $ z_j=1 $, then $ s_{ji}=0 $ for $ i\in[p+1,T] $.	
	
\end{proposition}

Define function $ [x]^{+} $ as $ [x]^{+}=\max\{0,x\} $, for $ x\in R $, then we get better expressions shown in the next proposition. We omit the proof because it is easy to testify the validity with the results of former propositions.
\begin{proposition}
	For the two-stage SLSCC problem, under Assumption 1, there exists an optimal production level of the form:
	\begin{eqnarray} 
	x_i &=& 
	\max_{1\le\tau\le m}\{\sum\limits_{t=i}^{T}d_{\tau t}[y_i-\sum\limits_{k=i+1}^{t}y_k-z_\tau]^{+}-s_{\tau(i-1)}\}, \quad i\in[p+1,T].
	\end{eqnarray}\\
	and an optimal inventory level of the form:
	\begin{eqnarray} 
	s_i &=& 
	\sum\limits_{t=i+1}^{p}d_t[1-\sum\limits_{k=i+1}^{t}y_k]^{+}+\max_{1\le\tau\le m}\{\sum\limits_{t=p+1}^{T} d_{\tau t}[1-\sum\limits_{k=i+1}^{t}y_k-z_\tau]^{+}\}, \quad i\in[1,p].
	\end{eqnarray}
	
	\begin{align} 
	s_{ji} &=&  
	\max_{1\le\tau\le m}\{\sum\limits_{t=p+1}^{i} d_{\tau t}(1-z_\tau-z_j)+\sum\limits_{t=i+1}^{T}d_{\tau t}[1-\sum\limits_{k=i}^{t}y_k-z_\tau-z_j]^{+}\}-\sum\limits_{t=p+1}^{i}d_{jt}(1-z_j), \quad i\in[p+1,T].
	\end{align}		
\end{proposition}

Using the relations $ x_i=d_i+s_i-s_{i-1},\ i\in[1,p] $, we can eliminate $ x_i,\ i\in[1,p] $, in the objective function, and with the result of Proposition 3, there is a natural way to construct a formulation which may be easier to compute, as follows (N-SLSCC):

\begin{align}
&\min\ \ \sum\limits_{i=1}^{p-1}(\alpha_i+h_i-\alpha_{i+1})s_i+(\alpha_p+h_p)s_p+\sum\limits_{i=1}^{p}\beta_i y_i+\sum\limits_{i=p+1}^{T}(\alpha_i x_i+\beta_i y_i)+\sum\limits_{j=1}^{m}(p_j\sum\limits_{i=p+1}^{T}h_i s_{ji})+\sum\limits_{i=1}^{p}\alpha_i d_i \nonumber \\
&s_i\ge \sum\limits_{t=i+1}^{\nu}d_t(1-\sum\limits_{k=i+1}^{t}y_k),\ \ i\in[1,p],\quad \nu\in [i+1,p] \\
&s_i\ge \sum\limits_{t=i+1}^{p}d_t(1-\sum\limits_{k=i+1}^{t}y_k)+\sum\limits_{t=p+1}^{\nu}d_{\tau t}(1-\sum\limits_{k=i+1}^{t}y_k-z_\tau),\quad 
i\in[1,p],\quad \tau\in\Omega,\quad \nu\in[p+1,T]\\
&x_i\ge \sum\limits_{t=i}^{\nu}d_{\tau t}(y_i-\sum\limits_{k=i+1}^{t}y_k-z_\tau)-s_{\tau(i-1)},\quad i\in[p+1,T],\quad \tau\in\Omega,\quad \nu\in[i,T]\\
&s_{ji}\ge \sum\limits_{t=p+1}^{i}d_{\tau t}(1-z_\tau-z_j)+\sum\limits_{t=i+1}^{\nu}d_{\tau t}(1-\sum\limits_{k=i}^{t}y_k-z_\tau-z_j)-\sum\limits_{t=p+1}^{i}d_{jt}(1-z_j), \nonumber \\
&\qquad\qquad\qquad\qquad\qquad\qquad\qquad\qquad\qquad\qquad i\in[p+1,T],\quad \tau,\ j\in\Omega,\quad \nu\in[i+1,T]\\
&\sum\limits_{j=1}^{m}p_j z_j\le \varepsilon\\
&x_i, s_i \in\mathbb{R}_+,\ s_{ji}\in \mathbb{R}_+, \ y_i\in\{0, 1\},z_j\in\{0,1\}
\end{align}

\begin{remark}
	In fact, we can omit $ s_i, s_{ji}\in \mathbb{R}_+ $ in constraint $ (27) $, because constraints $ (22)(23(25) $ with $ y_i,z_j\in\{0,1\} $ will insure $ s_i, s_{ji}\in \mathbb{R}_+ $ naturally by Proposition 3. However, when solving the linear relaxation of N-SLSCC, the optimal $ s_{ji} $ may be negative even when optimal $ z_j\ne 1 $ for some $ i\in [p+1,T] $, so we retain $ s_i, s_{ji}\in \mathbb{R}_+ $ in this formulation.
\end{remark}

\begin{proposition}
	Under Assumption 1, the optimal solutions of formulation N-SLSCC are also the optimal solutions of the original two-stage SLSCC problem.		
\end{proposition}

Now we obtain a mixed integer programming which can produce the optimal solutions of the original problem. However, not like many previous lot-sizing problems, maybe the new formulation N-SLSCC is not equivalent to its linear relaxation. While we can expect it more efficient than the original formulation because of the relation between N-SLSCC and the expressions of opimal solutions by Proposition 3. 
Except that, we found that a kind of subproblem can be solved in polynomial time, which we introduce as following.

\subsection{Description of $ S $-subproblem}
Let $ \mathcal{S} $ be the family of possible occurred scenarios set, i.e., $ \mathcal{S}=\{S\mid S\subseteq \Omega,\ \sum\limits_{j\in S}p_j(1-z_j)\ge 1-\varepsilon\} $, $ D_{ji} $ be the cumulant of demands from period $ p+1 $ to $ i $ of scenario $ j $, for $ i\in[p+1,T] $, $ j\in\Omega $, i.e., $ D_{ji}= \sum\limits_{t=p+1}^i d_{jt}$, $ d_i^S $ be the maximum among cumulants of demands from period $ p+1 $ to $ i $ of scenario $ j $, for $ j\in S $ and $ S\in\mathcal{S} $, i.e., $ d_i^S=\max\limits_{j\in S}D_{ji} $. By the definition of $ \mathcal{S} $, for every $ S\in \mathcal{S},\ z_j=0$, when $j\in S,\ z_j=1$, when $j\in \Omega\setminus S $ is a possible case that satisfies the chance constraint. Thereby, for any $ S\in \mathcal{S} $, we can define the related subproblem by restricting $ z_j=0 $ for $ j\in S $, and $ z_j=1 $ otherwise in the original deterministic equivalent formulation, which is called $ S $-subproblem. Let $ opt(*) $ be the optimal value of the original problem, and $ opt(S) $ be the optimal value of $ S $-subproblem, then it is easy to find $ opt(*)=\min\limits_{S\in\mathcal{S}}opt(S) $. Therefore, it is meaningful to study the character of $ S $-subproblem.

Define $\delta_{p+1}^S=d_{p+1}^S$, $\delta_i^S=d_i^S-d_{i-1}^S$, $i\in[p+2,T]$. Considering the conclusion of optimality condition of original problem, we can give a similar one for the $ S $-subproblem. 

\begin{proposition}
	For the $ S $-subproblem of two-stage SLSCC problem, under Assumption 1, there exists an optimal production level of the form:
	\begin{eqnarray} 
	x_i &=& 
	\sum\limits_{t=i}^T \delta_t^S[y_i-\sum\limits_{k=i+1}^t y_k]^+, \quad i\in[p+1,T].
	\end{eqnarray}\\
	and an optimal inventory level of the form:
	\begin{eqnarray} 
	s_i &=& 
	\sum\limits_{t=i+1}^{p}d_t[1-\sum\limits_{k=i+1}^{t}y_k]^{+}+\sum\limits_{t=p+1}^T \delta_t^S[1-\sum\limits_{k=i+1}^t y_k]^+, \quad i\in[1,p].
	\end{eqnarray}
	if $ j\in S $, then
	\begin{eqnarray} 
	s_{ji} &=&  
	\sum\limits_{t=p+1}^i \delta_t^S + \sum\limits_{t=i+1}^T \delta_t^S [1-\sum\limits_{k=i+1}^t y_k]^+ -D_{ji}, \quad i\in[p+1,T].
	\end{eqnarray}
	otherwise, $ s_{ji}=0 $.
\end{proposition}
\begin{proof}{Proof}
	We only need to prove Eq.(28)-(30). With $ z_j=0 $ for $ j\in S $ and $ z_j=1 $ otherwise, by Proposition 2,  for $ i\in[p+1,T] $ and $ y_i=1 $, we have $ x_i^S=d_{\psi(i)-1}^S-d_{i-1}^S=\sum\limits_{t=i}^{\psi(i)-1}(d_{t}^S-d_{t-1}^S)=\sum\limits_{t=i}^{\psi(i)-1}\delta_t^S=\sum\limits_{t=i}^T \delta_t^S[y_i-\sum\limits_{k=i+1}^t y_k]^+ $, then Eq.(28) holds. Observing that for $ i\in[1,p] $, $ \max\limits_{\tau\in S}\{\sum\limits_{t=p+1}^{T} d_{\tau t}[1-\sum\limits_{k=i+1}^{t}y_k]^{+}\}=\max\limits_{\tau\in S}D_{\tau(\psi(i)-1)}=d_{\psi(i)-1}^S=\sum\limits_{t=p+1}^{\psi(i)-1}\delta_t^S=\sum\limits_{t=p+1}^T \delta_t^S[1-\sum\limits_{k=i+1}^t y_k]^+ $, then by Proposition 3, Eq.(29) holds. For $ i\in[p+1,T] $ and $ j\in S $, also by Proposition 3, $ s_{ji}=\max\limits_{\tau\in S}\{\sum\limits_{t=p+1}^{i} d_{\tau t}+\sum\limits_{t=i+1}^{T}d_{\tau t}[1-\sum\limits_{k=i}^{t}y_k]^{+}\}-D_{ji}=\max\limits_{\tau\in S}D_{\tau(\psi(i)-1)}-D_{ji}=d_{\psi(i)-1}^S-D_{ji}=\sum\limits_{t=p+1}^{\psi(i)-1}\delta_t^S-D_{ji}=\sum\limits_{t=p+1}^i \delta_t^S+\sum\limits_{t=i+1}^T \delta_t^S[1-\sum\limits_{k=i+1}^t y_k]^+ -D_{ji} $, so Eq.(30) holds. 
	
\end{proof}

With the result of Proposition 5, an equivalent MIP formulation of the $ S $-subproblem can be described as:
\begin{align}
&\min\sum\limits_{i=1}^{p-1}(\alpha_i+h_i-\alpha_{i+1})s_i^S+(\alpha_p+h_p)s_p^S+\sum\limits_{i=1}^{p}\beta_i y_i+\sum\limits_{i=p+1}^{T}(\alpha_i x_i^S+\beta_i y_i)+\sum\limits_{j=1}^{m}(p_j\sum\limits_{i=p+1}^{T}h_i s_{ji}^S)+\sum\limits_{i=1}^{p}\alpha_i d_i\nonumber\\
&s_i^S\ge \sum\limits_{t=i+1}^\nu d_t(1-\sum\limits_{k=i+1}^t y_k), \quad i\in [1,p], \quad \nu\in [i+1,p]\\
&s_i^S\ge \sum\limits_{t=i+1}^p d_t(1-\sum\limits_{k=i+1}^t y_k) +\sum\limits_{t=p+1}^\nu \delta_t^S(1-\sum\limits_{k=i+1}^t y_k), \quad i\in [1,p],\quad \nu\in [p+1,T]\\
&x_i^S\ge\sum\limits_{t=i}^\nu \delta_t^S (y_i-\sum\limits_{k=i+1}^t y_k), \quad i\in [p+1,T], \quad \nu\in[i,T]\\
&s_{ji}^S\ge\sum\limits_{t=p+1}^i \delta_t^S+\sum\limits_{t=i+1}^\nu \delta_t^S(1-\sum\limits_{k=i+1}^t y_k)-D_{ji}, \quad i\in [p+1,T],\quad \nu\in[i+1,T], \quad j\in S\\
&s_{ji}^S=0,\quad i\in [p+1,T],\quad j\in \Omega\setminus S\\
&y_i\in\{0,1\},\quad i\in N
\end{align}

Observing that for any $ S\in \mathcal{S} $, $ x_i^S=d_{ji}+s_{ji}^S -s_{j(i-1)}^S,\ i\in[p+1,T],\ j\in S $ (assume $ s_{jp}^S=s_p^S, j\in\Omega $), therefore, we obtain $ x_i^S=\frac{1}{\sum\limits_{j\in S}p_j}[\sum\limits_{j \in S }p_j(d_{ji}+s_{ji}^S-s_{j(i-1)}^S)],\ i\in[p+1,T] $. Substitute former equations into the objective function , then we get a simpler but equivalent formulation as:
\begin{align}
&\min \sum\limits_{i=1}^{T}\beta_i y_i+\sum\limits_{i = 1}^p h_i^\prime s_i^S + \sum\limits_{i = p + 1}^T \sum\limits_{j=1}^m h_{ji}^S s_{ji}^S + \sum\limits_{i = p + 1}^T r_i^S \nonumber\\
&s_i^S\ge \sum\limits_{t=i+1}^\nu d_t(1-\sum\limits_{k=i+1}^t y_k),\quad i\in [1,p],\quad \nu\in[i+1,p]\\
&s_i^S\ge \sum\limits_{t=i+1}^p d_t(1-\sum\limits_{k=i+1}^t y_k) +\sum\limits_{t=p+1}^\nu \delta_t^S(1-\sum\limits_{k=i+1}^t y_k), \quad i\in [1,p],\quad \nu\in[p+1,T]\\
&s_{ji}^S\ge \sum\limits_{t=p+1}^i \delta_t^S+\sum\limits_{t=i+1}^\nu \delta_t^S(1-\sum\limits_{k=i+1}^t y_k) -D_{ji}, \quad i\in [p+1,T],\quad j\in S, \quad \nu\in [i+1,T]\\
&s_{ji}^S=0,\quad i\in [p+1,T],\quad j\in \Omega\setminus S\\
&y_i\in\{0,1\},\quad i\in N
\end{align}
where for any $ S\in\mathcal{S} $\\
\begin{center}
	$ h_i^\prime =\alpha_i+h_i-\alpha_{i+1},\quad i\in[p+1,T-1],\quad j\in S .$\\
	\begin{eqnarray} 
	h_{ji}^S &=& 
	\left\{ \begin{array}{lll} 
	\frac{p_j}{\sum\limits_{j\in S}p_j}(\alpha_i+\sum\limits_{j\in S}p_j h_i-\alpha_{i+1}), \quad i\in[p+1, T-1],\quad j\in S \\
	\frac{p_j}{\sum\limits_{j\in S}p_j}(\alpha_T+\sum\limits_{j\in S}p_j h_T), \quad i=T,\quad j\in S\\
	p_j h_i, \quad i\in[p+1, T],\quad j\in \Omega\setminus S.
	\end{array} 
	\right. 
	\end{eqnarray} \\
	\begin{eqnarray}
	r_i^S&=&
	\left\{ \begin{array}{ll}
	\alpha_i d_i,\quad i\in[1,p]\\
	\frac{\alpha_i}{\sum\limits_{j\in S}p_j}\sum\limits_{j\in S}p_jd_{ji},\quad i\in [p+1, T].
	\end{array}
	\right.
	\end{eqnarray}	
\end{center}
As for any $ S\in \mathcal{S} $, $ \sum\limits_{j\in S}p_j\ge 1-\varepsilon $, then according to Assumption 1 we know all $ h_i^\prime $, $ h_{ji}^S $, $ r_i^S $ are nonnegative. 

In fact, the $ S $-subproblem has a very good performance, that is, the optimal solutions of its linear relaxition are the optimal solutions of itself as well. We will prove the conclusion through an extended formulation.

Let $ u_{it}=1 $ if $ s_i $ contains $ d_t $ for $ i+1\le t\le p $ or $ s_i $ contains $ \delta_t^S $ for $ i\le p<t\le T $ or $ s_{ji} $ cantains $ \delta_t^S $ for $ p+2\le i+1\le t\le T $, and $ u_{it}=0 $ otherwise. We consider an extended formulation
\begin{align}
&\min \sum\limits_{i=1}^{T}\beta_i y_i+\sum\limits_{i = 1}^p h_i^\prime s_i^S + \sum\limits_{i = p + 1}^T \sum\limits_{j=1}^m h_{ji}^S s_{ji}^S + \sum\limits_{i = p + 1}^T r_i^S\\
&s_i^S=\sum\limits_{t=i+1}^p d_t u_{it} +\sum\limits_{t=p+1}^T \delta_t^S u_{it}, \quad i\in [1,p]\\
&s_{ji}^S=\sum\limits_{t=p+1}^i \delta_t^S+\sum\limits_{t=i+1}^T \delta_t^S u_{it} -D_{ji}, \quad i\in [p+1,T],\quad j\in S\\
&s_{ji}^S=0,\quad i\in [p+1,T],\quad j\in \Omega\setminus S\\
&u_{it}\ge 1-\sum\limits_{k=i+1}^t y_k,\quad i\in N,\quad t\in[i+1,T]\\
&u_{it}\ge0,\quad 0\le y_i\le 1,\quad i\in N,\quad t\in[i+1,T]\\
&y_i\ \ integer,\quad i\in N
\end{align}

As the study of uncapacitated lot-sizing problem with Wager-Whitin costs in \citet*{8Pochet1994}, we can get a similar conclusion as follows:
\begin{proposition}
	The constraint matrix corresponding to the constraints (48) (49) is totally unimodular. Then the linear program (44)-(49) is an extended formulation for $ S $-subproblem.
\end{proposition}

The proof of Proposition 6 is trivial, see the proof of Proposition 2 in \citet*{8Pochet1994} for details.

Define polyhedron $ P^S=\{(s^S,s_j^S,y)|(s^S,s_j^S,y) \ satisfies \ (37)-(40)\} $. We now consider the projection of the polyhedron $ Q^S=\{(s^S,s_j^S,y,u)|(s^S,s_j^S,y,u) \ satisfies \ (45)-(49)\} $.     
\begin{theorem}
	$ Proj_{(s^S, s_j^S, y)} Q^S=P^S $. Then polyhedron $ P^S $ is integral.
\end{theorem}
\begin{proof}{Proof}
	We project the polyhedron $ Q^S $ onto the $ (s^S,s_j^S,y) $ space. It is obvious that the extreme points of $ Q^S $ all satisfy $ u_{it}=[1-\sum\limits_{k=i+1}^t y_k]^+ $ and $ u_{it}\ge u_{i(t+1)} $, for $ i\in N $, thus $ s_i^S=\sum\limits_{t=i+1}^p d_t[1-\sum\limits_{k=i+1}^t y_k]^+ +\sum\limits_{t=p+1}^T \delta_t^S[1-\sum\limits_{k=i+1}^t y_k]^+ $	and $ s_{ji}^S=\sum\limits_{t=p+1}^i \delta_t^S+\sum\limits_{t=i+1}^T \delta_t^S[1-\sum\limits_{k=i+1}^t y_k]^+ -D_{ji} $, which are equivalent to $ s_i^S=\max\{\max\limits_{i+1\le \nu\le p}\{\sum\limits_{t=i+1}^\nu d_t(1-\sum\limits_{k=i+1}^t y_k)\}, \max\limits_{p+1\le \nu\le T}\{\sum\limits_{t=i+1}^p d_t(1-\sum\limits_{k=i+1}^t y_k) +\sum\limits_{t=p+1}^\nu \delta_t^S(1-\sum\limits_{k=i+1}^t y_k)\}\} $ and $ s_{ji}^S=\max\limits_{i+1\le \nu\le T}\{\sum\limits_{t=p+1}^i \delta_t^S+\sum\limits_{t=i+1}^\nu \delta_t^S(1-\sum\limits_{k=i+1}^t y_k) -D_{ji}\} $. Hence, the extreme points of $ Proj_{(s^S, s_j^S, y)} Q^S$ and $ P^S $ correspond. Denote the set of the extreme points as $ V $. In addition, it is trivial that $ Proj_{(s^S, s_j^S, y)} Q^S$ and $ P^S $ have the same recession cone, then apparently we have $ Proj_{(s^S, s_j^S, y)} Q^S=\{(s^S,s_{j}^S,y)\mid (\bar{s}^S,\bar{s}_j^S,y)\in conv(V),s^S\ge \bar{s}^S,s_j^S\ge\bar{s}_j^S \}=P^S $. By Proposition 6, $ P^S $ is integral.
\end{proof}

Theorem 1 means that optimizing over polyhedron $ P^S $ is enough to solve $ S $-subproblem. Therefore, we can solve the original (SLSCC) problem by solving at most $ card(\mathcal{S}) $ LP subproblems. Let $ \{\langle 1\rangle , \langle 2 \rangle, \ldots, \langle m\rangle\} $ be a permutation of set $ \Omega $ with $ p_{\langle 1\rangle}\le p_{\langle 2\rangle}\le \cdots \le p_{\langle m\rangle} $, and define parameter $ \kappa $ as the integer such that $ \sum\limits_{i=1}^\kappa p_{\langle i\rangle}\le \varepsilon $ and $ \sum\limits_{i=1}^{\kappa+1} p_{\langle i\rangle}> \varepsilon $. When $ m $ is fixed, the set of occurred scenarios sets $ \mathcal{S} $ has at most $ \sum\limits_{k=1}^\kappa \binom{k}{m} $ elements. Hence, the original (SLSCC) problem can be solved in polynomial time with fixed $ m $. However, for variable $ m $, it remains further study to sovle (SLSCC) more efficiently. Maybe an algorithm using the character of  LP $ S $-subproblems should be considered.

\begin{remark}
	We need not to involve all the possible occurred scenarios sets in $ \mathcal{S} $, in fact, if there are two elements $ S_1, S_2\in \mathcal{S} $ and $ S_1 $ is a proper subset of $ S_2 $, i.e. $ S_1\subsetneq S_2 $, then we can eliminate $ S_2 $ from $ \mathcal{S} $ without affection on the optimal solution. Therefore, there is no matter to assume that any two elements of $ \mathcal{S} $ do not have inclusion relation.
\end{remark}

\section{A Branch-and-bound Algorithm}
In this section, we develop a specialized branch-and-bound algorithm to solve the two-stage SLSCC problems exploiting the property of $ S $-subproblem and the formulation N-SLSCC. The algorithm is described in Algorithm 1. In general, the algorithm recursively branch the indicator variables $ z_j $ in the formulation N-SLSCC, and solve $ S $-subprobem to obtain a feasible solution and upper bound to reduce branches. We provide its detailed description next.

\begin{algorithm}
	\caption{A branch-and-bound algorithm for N-SLSCC}
	\begin{algorithmic}[1]
		\Require
		\State solve $ LR(\emptyset) $ and obtain its optimal value $ LB(\emptyset) $ and solution $ ({x^2}^0,{y^1}^0,{y^2}^0,{s^1}^0,{s^2_j}^0,z^0) $
		\If {$ ({y^1}^0,{y^2}^0,z^0)\in\{0,1\}^{T+m} $}
		\State STOP $ ({x^2}^0,{y^1}^0,{y^2}^0,{s^1}^0,{s^2_j}^0,z^0) $ is an optimal solution
		\Else
		\State set $ \mathcal{L}=\{\emptyset\} $, $ UB=+\infty $, and $ z^*=\emptyset $; give the tolerance $ \Delta $
		\EndIf 	
	\end{algorithmic}
	
	\begin{algorithmic}[1]
		\Ensure
		\While {$ \mathcal{L}\ne \emptyset $}
		\State 	select $ \mathcal{C}\in\mathcal{L} $ such that $ LB(\mathcal{C})=\min_{\mathcal{C^\prime}\in\mathcal{L}}\{LB(\mathcal{C^\prime})\} $
		\If {$ UB-LB(\mathcal{C})\le\Delta $}
		\State STOP the optimal solution of $ S(z^*) $-subproblem with $ z^* $ is an global optimal solution
		\EndIf
		\State solve $ S(z^{\mathcal{C}}) $-subproblem and obtain its optimal value and solution
		\If {$ UB(z^{\mathcal{C}})<UB $}
		\State $ UB \gets UB(z^{\mathcal{C}}) $ and $ z^*\gets {z^{\mathcal{C}}}_I $
		\EndIf
		\If {$ UB-LB(\mathcal{C})\le\Delta $}
		\State STOP the optimal solution of $ S(z^{\mathcal{C}}) $-subproblem with $ z^* $ is an global optimal solution
		\Else
		\State branch $ \mathcal{C} $ into $ \mathcal{C}_1=\mathcal{C}\cup\{\bar{j(z)}\} $ and $ \mathcal{C}_2=\mathcal{C}\cup\{\hat{j(z)}\} $, set $ \mathcal{L}\gets\mathcal{L}\setminus\{\mathcal{C}\} $
		\For {$ i=1,2 $} 
		\State solve $ LR(\mathcal{C}_i) $
		\If {$ LR(\mathcal{C}_i) $ is feasible}
		\State set $ \mathcal{L}\gets\mathcal{L}\cup\{\mathcal{C}_i\} $
		\EndIf
		\EndFor
		\EndIf
		\For {each $ \mathcal{C}\in\mathcal{L} $}
		\If {$ LB(\mathcal{C})>UB $}
		\State fathom $ \mathcal{C} $, set $ \mathcal{L}\gets \mathcal{L}\setminus\{\mathcal{C}\} $
		\EndIf
		\EndFor
		\EndWhile
	\end{algorithmic}
\end{algorithm}

In the following description, for any $ j\in \Omega $, let $ \bar{j} $ represent constraint $ z_j=0 $ and $ {\bf{\Omega}}_0=\{\bar{j}\mid j\in \Omega\} $ be the set of all such kind of constraint; let $ \hat{j} $ represent constraint $ z_j=1 $ and $ {\bf{\Omega}}_1=\{\hat{j}\mid j\in \Omega\} $ be the set of all such kind of constraint; $ \mathcal{C} $ denotes a subset of $ {\bf{\Omega}}_0\cup{\bf{\Omega}}_1 $, which does not include $ \bar{j} $ and $ \hat{j} $ at the same time for any $ j\in\Omega $; $ LR(\mathcal{C}) $ denotes the linear relaxation of $ \mathcal{C} $-subproblem, which is defined by $ (51)-(62) $ if there exsits any $ \bar{j}\in \mathcal{C} $ and by N-SLSCC with constraints in $ \mathcal{C} $ otherwise; $ LB(\mathcal{C}) $ denotes the optimal value of problem $ LR(\mathcal{C}) $, which is the lower bound on the optimal value of N-SLSCC over $ \mathcal{C} $; $ UB $ denotes a global upper bound on the optimal value; $ \mathcal{L} $ is a list of un-fathomed subsets of $ {\bf{\Omega}}_0\cup{\bf{\Omega}}_1 $ defined formerly; $ z^* $ denotes the best candidate indicator vector; $ z^{\mathcal{C}} $ denotes the optimal indicator vector of problem $ LR(\mathcal{C}) $; for $ z\in[0,1]^m $, let $ \{z_{\langle 1\rangle}, z_{\langle 2\rangle} ,\ldots, z_{\langle m\rangle}\} $ be a permutation of set $ \{z_1, z_2, \ldots, z_m\} $, which satisfies $ z_{\langle 1\rangle}\le z_{\langle 2\rangle} \le \ldots\le z_{\langle m\rangle} $, define $ S(z)=\{\langle 1\rangle, \langle 2\rangle,\ldots,\langle \kappa\rangle\mid \sum\limits_{i=1}^{\kappa-1} p_{\langle i\rangle}< 1-\varepsilon\ and \ \sum\limits_{i=1}^\kappa p_{\langle i\rangle}\ge 1-\varepsilon,\ \kappa\in\Omega\} $ as an index subset of $ \Omega $, then obviously $ S(z)\in \mathcal{S} $, and $ UB(z) $ denotes the optimal value of $ S(z) $-subproblem; for $ z\in[0,1]^m $, define $ z_I\in\{0,1\}^m $ which satisfies $ {z_I}_j=0 $ when $ j\in S(z) $ and $ {z_I}_j=1 $ otherwise, and $ j(z)=\argmin\limits_{\langle j\rangle\in\Omega}\{z_{\langle j\rangle}\mid 0<z_{\langle j\rangle}<1,\langle j\rangle\in\Omega\} $; let $ \Delta $ be the tolerance of optimal value, which can avoid excessive computational cost to reduce a small difference to achieve the optimal value. 

For each set $ \mathcal{C} $, let $ J^\mathcal{C}_1=\{j\mid \bar{j}\in\mathcal{C}\} $ and $ J^\mathcal{C}_2=\{j\mid \hat{j}\in\mathcal{C}\} $, respectively. Define index set $ \Omega(\mathcal{C})=\Omega\setminus({J^\mathcal{C}_1}\cup{J^\mathcal{C}_2}) $. Similar to the definition of $ S $-subproblem, we can define a kind of $ \mathcal{C} $-subproblem when $ J^\mathcal{C}_1\ne \emptyset $ and utilizing the relation $ x_i=\frac{1}{\sum\limits_{j\in J^\mathcal{C}_1}p_j}[\sum\limits_{j \in J^\mathcal{C}_1 }p_j(d_{ji}+s_{ji}-s_{j(i-1)})],\ i\in[p+1,T] $ as follows:
\begin{align}
&\min \sum\limits_{i=1}^{T}\beta_i y_i+\sum\limits_{i = 1}^p h_i^\prime s_i + \sum\limits_{i = p + 1}^T \sum\limits_{j=1}^m h_{ji}^{J^\mathcal{C}_1} s_{ji} + \sum\limits_{i = p + 1}^T r_i^{J^\mathcal{C}_1}\\
&s_i\ge \sum\limits_{t=i+1}^\nu d_t(1-\sum\limits_{k=i+1}^t y_k), \quad i\in [1,p], \quad \nu\in [i+1,p]\\
&s_i\ge \sum\limits_{t=i+1}^p d_t(1-\sum\limits_{k=i+1}^t y_k) +\sum\limits_{t=p+1}^\nu \delta_t^{J^\mathcal{C}_1}(1-\sum\limits_{k=i+1}^t y_k), \quad i\in [1,p],\quad \nu\in [p+1,T]\\
&s_i\ge \sum\limits_{t=i+1}^{p}d_t(1-\sum\limits_{k=i+1}^{t}y_k)+\sum\limits_{t=p+1}^{\nu}d_{\tau t}(1-\sum\limits_{k=i+1}^{t}y_k-z_\tau),\quad 
i\in[1,p],\quad \tau\in\Omega(\mathcal{C}),\quad \nu\in[p+1,T]\\
&s_{ji}\ge\sum\limits_{t=p+1}^i \delta_t^{J^\mathcal{C}_1}+\sum\limits_{t=i+1}^\nu \delta_t^{J^\mathcal{C}_1}(1-\sum\limits_{k=i+1}^t y_k)-D_{ji}, \quad i\in [p+1,T],\quad \nu\in[i+1,T], \quad j\in {J^\mathcal{C}_1}\\
&s_{ji}\ge \sum\limits_{t=p+1}^{i}d_{\tau t}(1-z_\tau)+\sum\limits_{t=i+1}^{\nu}d_{\tau t}(1-\sum\limits_{k=i}^{t}y_k-z_\tau)-D_{ji}, \nonumber \\
&\qquad\qquad\qquad\qquad\qquad\qquad\qquad\qquad i\in[p+1,T],\quad \nu\in[i+1,T], \quad \tau\in\Omega(\mathcal{C}),\quad j\in{J^\mathcal{C}_1} \\
&s_{ji}\ge\sum\limits_{t=p+1}^i \delta_t^{J^\mathcal{C}_1}(1-z_j)+\sum\limits_{t=i+1}^\nu \delta_t^{J^\mathcal{C}_1}(1-\sum\limits_{k=i+1}^t y_k-z_j)-D_{ji}(1-z_j), \nonumber \\
&\qquad\qquad\qquad\qquad\qquad\qquad\qquad\qquad i\in [p+1,T],\quad \nu\in[i+1,T], \quad j\in \Omega(\mathcal{C})\\
&s_{ji}\ge \sum\limits_{t=p+1}^{i}d_{\tau t}(1-z_\tau-z_j)+\sum\limits_{t=i+1}^{\nu}d_{\tau t}(1-\sum\limits_{k=i}^{t}y_k-z_\tau-z_j)-D_{ji}(1-z_j), \nonumber \\
&\qquad\qquad\qquad\qquad\qquad\qquad\qquad\qquad i\in[p+1,T],\quad \nu\in[i+1,T], \quad \tau\in\Omega(\mathcal{C}),\quad j\in\Omega(\mathcal{C}) \\
&s_{ji}=0,\quad i\in [p+1,T],\quad j\in{J^\mathcal{C}_2}\\
&\sum\limits_{j\in\Omega(\mathcal{C})}p_j z_j\le \varepsilon-\sum\limits_{j\in{J^\mathcal{C}_2}}p_j\\
&z_j=0, \quad j\in {J^\mathcal{C}_1}, \quad z_j=1, \quad j\in {J^\mathcal{C}_2}\\
&s_i \in\mathbb{R}_+,\ s_{ji}\in \mathbb{R}_+, \ y_i\in\{0, 1\},z_j\in\{0,1\}
\end{align}
where $ h_i^\prime $, $ h_{ji}^{J^\mathcal{C}_1} $, $ r_i^{J^\mathcal{C}_1} $ and $ \delta_t^{J^\mathcal{C}_1} $ are just the same as those defined in the $ S $-subproblem when $ S=J^\mathcal{C}_1 $ omitting the requirement $ S\in\mathcal{S} $. Obviously, $ \mathcal{C} $-subproblem has the same optimal solutions as N-SLSCC with constraints in $ \mathcal{C} $. Now we can describe the branch-and-bound scheme below.

The algorithm starts by solving the linear relaxation of N-SLSCC $ LR(\emptyset) $, which gives the lower bound $ LB(\emptyset) $, then we check if the optimal solution is a feasible solution of N-SLSCC, if so, it is also the optimal solution of N-SLSCC and the algorithm ends; if not, then the algorithm starts its main loop.

The completion of branching is by adding $ \bar{j} $ or $ \hat{j} $ type of constraint into an element of the list $ \mathcal{L} $. First, we select a branching element $ \mathcal{C} $ from $ \mathcal{L} $, which is based upon the least-lower bound rule to guarantee the bounding process is bound improving (line 2 of the main loop in Algorithm 1). After the selection, suppose $ z $ is the optimal indicator vector when $ LR(\mathcal{C}) $ attains its optimal value, then solve $ S(z) $-subproblem to obtain a feasible solution and an upper bound $ UB(z) $ of original problem. Let $ UB=UB(z) $ if $ UB>UB(z) $, and $ UB=UB $ otherwise. If $ LB(\mathcal{C})<UB $, then branch $ \mathcal{C} $ into two new set $ \mathcal{C}_1 $ and $ \mathcal{C}_2 $ and add $ \mathcal{C}_i $ into list $ \mathcal{L} $ when problem $ LR(\mathcal{C}_i) $ is  feasible for $ i=1,2 $, as the same time eliminate $ \mathcal{C} $ from $ \mathcal{L} $, where $ \mathcal{C}_1=\mathcal{C}\cup\{\bar{j(z)}\} $ and $ \mathcal{C}_2=\mathcal{C}\cup\{\hat{j(z)}\} $.

The reduction of branches is very natural, after branching, for every element $ \mathcal{C}\in \mathcal{L} $, examine if $ LB(\mathcal{C})>UB $, if so, eliminate it from list $ \mathcal{L} $, otherwise preserve it. 

The algorithm stops when the difference between the global upper bound and lower bound is no more than the given tolerance, then $ z^* $ is the optimal indicator vector and $ S(z^*) $-subproblem gives the optimal solution and optimal value of N-SLSCC (The optimal production levels $ x^2 $ can be computed by inventory levels $ s^1 $ and $ s_j^2 $).

We can prove that the branch-and-bound algorithm for N-SLSCC converges in finitely many steps.

\begin{theorem}
	For the two-stage SLSCC problem, under Assumption 1, its N-SLSCC formulation can be solved by Algorithm 1 in finitely many steps even when the tolerence $ \Delta=0 $.
\end{theorem}
\begin{proof}{Proof}
	When the algorithm terminates before starting the main loop, obviously the conclusion holds. Then we consider the algorithm terminates during the main loop. For every $ \mathcal{C}\in\mathcal{L} $, it can contain at most $ m $ elements, and for each iteration, it can easily be seen that the chosen $ \mathcal{C}\in\mathcal{L} $ will be branched into two sets and eliminated from $ \mathcal{L} $, and both of them contain one more element than $ \mathcal{C} $, if the algorithm does not stop. Thus, after at most $ 2^m $ iterations, every $ \mathcal{C}\in\mathcal{L} $ will contain $ m $ elements and can not be branched anymore. By the definition of $ \mathcal{C} $-subproblem, it is equivalent to $ J^\mathcal{C}_1 $-subproblem with $ J^\mathcal{C}_1\in\mathcal{S} $ when $ card(\mathcal{C})=m $. Therefore, after at most $ 2^m $ iterations, for the set $ \mathcal{C}\in\mathcal{L} $ such that $ LB(\mathcal{C})=\min_{\mathcal{C^\prime}\in\mathcal{L}}\{LB(\mathcal{C^\prime})\} $, solving $ LR(\mathcal{C}) $ will just obtain the global optimal solution. Hence, the conclusion holds.
\end{proof}

\begin{APPENDIX}{}
	\textbf{Proof of Proposition 1:} Firstly, we prove that there is an optimal $ x_i $ satisfies (13), then using (13) to prove equations (14) (15). For the simplicity of expression, we define that $ y_{T+1}=1 $ and $ \alpha_{T+1}=\beta_{T+1}=h_{T+1}=0 $.
	
	If an optimal solution' $ x_i $ does not have the above form, then there must be some $ x_k, k\in N $ violates (13) . We can analyze that in three conditons.
	
	1. Assume that $ x_k>0 $ and $ x_k\ne\sum\limits_{t=k}^{\psi (k)-1}d_t, \psi(k)\le p $, and for any $ x_q>0, q<k $, there is $ x_q=\sum\limits_{t=q}^{\psi(q)-1}d_t $.
	Let functions $ g_1 $ and $ f $ respectively be
	\begin{equation*}
	g_1(x,y,s,s_j,z)=\sum\limits_{i\in[1,p]\setminus[k,\psi(k)]} (\alpha_i x_i+\beta_i y_i+h_i s_i)+\sum\limits_{i = p + 1}^T(\alpha_i x_i+\beta_i y_i)+\sum\limits_{j\in{J_z}}(p_j \sum\limits_{i=p+1}^T h_i s_{ji}) 				
	\end{equation*}
	\begin{equation*}
	\begin{aligned}
	f(x,y,s,s_j,z)&=g_1(x,y,s,s_j,z)+(\alpha_k x_k+\beta_k y_k+h_k s_k)+\sum\limits_{i=k+1}^{\psi(k)-1}h_i s_i\\
	&+(\alpha_{\psi(k)} x_{\psi(k)}+\beta_{\psi(k)} y_{\psi(k)}+h_{\psi(k)} s_{\psi(k)}) 
	\end{aligned}
	\end{equation*}
	
	case 1.1: If $  x_k<\sum\limits_{t=k}^{\psi (k)-1}d_t $, then the demand of period $ \psi(k)-1 $ is not satisfied, contradiction.
	
	case 1.2: If $ x_k>\sum\limits_{t=k}^{\psi (k)-1}d_t $, then let $ (\bar{x}, \bar{y}, \bar{s}, \bar{s}_j, \bar{z}) $ be \\
	$ \bar{x}_k=x_k-\epsilon $, $ \bar{s}_i=s_i-\epsilon $ for any $ i\in[k,\psi(k)-1] $, and $ \bar{x}_{\psi(k)}=x_{\psi(k)}+\epsilon $, and other components are the same as $ (x,y,s,s_j,z) $.
	
	Then
	\begin{equation*}
	\begin{aligned}
	f(\bar{x},\bar{y},\bar{s},\bar{s}_j,z)&=g_1(x,y,s,s_j,z)+(\alpha_k (x_k-\epsilon)+\beta_k y_k+h_k (s_k-\epsilon))+\sum\limits_{i=k+1}^{\psi(k)-1}h_i (s_i-\epsilon)\\
	&+(\alpha_{\psi(k)} (x_{\psi(k)}+\epsilon)+\beta_{\psi(k)} y_{\psi(k)}+h_{\psi(k)} s_{\psi(k)})\\
	&=f(x,y,s,s_j,z)-(\alpha_k+\sum\limits_{i=k}^{\psi(k)-1} h_i-\alpha_{\psi(k)})\epsilon\\
	&=f(x,y,s,s_j,z)-\sum\limits_{i=k}^{\psi(k)-1}(\alpha_i+h_i-\alpha_{i+1})\epsilon
	\end{aligned}
	\end{equation*}
	by Assumption 1, we have $ \sum\limits_{i=k}^{\psi(k)-1}(\alpha_i+h_i-\alpha_{i+1})\ge 0 $, if $ \sum\limits_{i=k}^{\psi(k)-1}(\alpha_i+h_i-\alpha_{i+1})> 0 $, then $ f(\bar{x},\bar{y},\bar{s},\bar{s}_j,\bar{z})< f(x,y,s,s_j,z) $, for any $ 0<\epsilon\le x_k-\sum\limits_{t=k}^{\psi(k)-1}d_t $, which contradicts to the optimal property of $ (x,y,s,s_j,z) $. Therefore, $ \sum\limits_{i=k}^{\psi(k)-1}(\alpha_i+h_i-\alpha_{i+1})= 0 $, then $ \epsilon $ can be increased such that $ \bar{x}_k=\sum\limits_{t=k}^{\psi(k)-1}d_t $ without affecting the optimal property.
	
	2. Let $ l=argmax\{i: y_i=1, i\le p\} $, if $ x_l\ne \sum\limits_{t=l}^p d_t+\max\limits_{j\in J_z}\{\sum\limits_{t=p+1}^{\psi(l)-1}d_{jt}\} $, then there are two cases.
	
	case 2.1 similar to case 1.1.
	
	case 2.2 If $ x_l> \sum\limits_{t=l}^p d_t+\max\limits_{j\in J_z}\{\sum\limits_{t=p+1}^{\psi(l)-1}d_{jt}\} $, then let functions $ g_2 $ and $ f $ respectively be 
	\begin{equation*}
	g_2(x,y,s,s_j,z)=\sum\limits_{i=1}^{l-1} (\alpha_i x_i+\beta_i y_i+h_i s_i)+\sum\limits_{i = \psi(l)+1}^T(\alpha_i x_i+\beta_i y_i)+\sum\limits_{j\in{J_z}}(p_j \sum\limits_{i=\psi(l)+1}^T h_i s_{ji}) 				
	\end{equation*}
	\begin{equation*}
	\begin{aligned}
	f(x,y,s,s_j,z)&=g_2(x,y,s,s_j,z)+(\alpha_l x_l+\beta_l y_l+h_l s_l)+\sum\limits_{i=l+1}^{p}h_i s_i+\sum\limits_{j\in{J_z}}(p_j\sum\limits_{i = p + 1}^{\psi(l)-1}h_is_{ji})\\
	&+(\alpha_{\psi(l)} x_{\psi(l)}+\beta_{\psi(l)} y_{\psi(l)}+\sum\limits_{j\in J_z}p_jh_{\psi(i)}s_{j\psi(i)}) 
	\end{aligned}
	\end{equation*}
	and let $ (\bar{x}, \bar{y}, \bar{s}, \bar{s}_j, \bar{z}) $ be \\
	$ \bar{x}_l=x_l-\epsilon $, $ \bar{s}_i=s_i-\epsilon $ for any $ i\in[l,p] $, $ \bar{s}_{ji}=s_{ji}-\epsilon $ for any $ i\in [p+1, \psi(l)-1], j\in J_z $ and $ \bar{x}_{\psi(l)}=x_{\psi(l)}+\epsilon $, and other components are the same as $ (x,y,s,s_j,z) $.
	
	Then we have 
	\begin{equation*}
	\begin{aligned}
	f(\bar{x},\bar{y},\bar{s},\bar{s}_j,z)&=g_2(x,y,s,s_j,z)+(\alpha_l (x_l-\epsilon)+\beta_l y_l+h_l (s_l-\epsilon))+\sum\limits_{i=l+1}^{p}h_i (s_i-\epsilon)\\
	&+\sum\limits_{j\in J_z}[p_j\sum\limits_{i = p + 1}^{\psi(l)-1}h_i(s_{ji}-\epsilon)]+(\alpha_{\psi(l)} (x_{\psi(l)}+\epsilon)+\beta_{\psi(l)} y_{\psi(l)}+\sum\limits_{j\in J_z}p_j h_{\psi(l)} s_{j\psi(l)})\\
	&=f(x,y,s,s_j,z)-(\alpha_l+\sum\limits_{i=l}^{p} h_i+\sum\limits_{j\in J_z}p_j\sum\limits_{i = p + 1}^{\psi(l)-1}h_i-\alpha_{\psi(l)})\epsilon\\
	&=f(x,y,s,s_j,z)-[\sum\limits_{i=l}^{p}(\alpha_i+h_i-\alpha_{i+1})+\sum\limits_{i = p + 1}^{\psi(l)-1}(\alpha_i+\sum\limits_{j\in J_z}p_j h_i-\alpha_{i+1})]\epsilon
	\end{aligned}
	\end{equation*}
	by Assumption 1, we have $ \sum\limits_{i=l}^{p}(\alpha_i+h_i-\alpha_{i+1})+\sum\limits_{i = p + 1}^{\psi(l)-1}(\alpha_i+\sum\limits_{j\in J_z}p_j h_i-\alpha_{i+1})\ge \sum\limits_{i=l}^{p}(\alpha_i+h_i-\alpha_{i+1})+\sum\limits_{i = p + 1}^{\psi(l)-1}(\alpha_i+(1-\epsilon) h_i-\alpha_{i+1})\ge 0 $, if $ \sum\limits_{i=l}^{p}(\alpha_i+h_i-\alpha_{i+1})+\sum\limits_{i = p + 1}^{\psi(l)-1}(\alpha_i+\sum\limits_{j\in J_z}p_j h_i-\alpha_{i+1})>0 $, then $ f(\bar{x},\bar{y},\bar{s},\bar{s}_j,\bar{z})< f(x,y,s,s_j,z) $, for any $ 0<\epsilon\le x_l-(\sum\limits_{t=l}^p d_t+\max\limits_{j\in J_z}\{\sum\limits_{t=p+1}^{\psi(l)-1}d_{jt}\}) $, which contradicts to the optimal property of $ (x,y,s,s_j,z) $. Therefore, $ \sum\limits_{i=l}^{p}(\alpha_i+h_i-\alpha_{i+1})+\sum\limits_{i = p + 1}^{\psi(l)-1}(\alpha_i+\sum\limits_{j\in J_z}p_j h_i-\alpha_{i+1})=0 $, then $ \epsilon $ can be increased such that $ \bar{x}_l=\sum\limits_{t=l}^p d_t+\max\limits_{j\in J_z}\{\sum\limits_{t=p+1}^{\psi(l)-1}d_{jt}\} $ without affecting the optimal property.
	
	3. If there is some $ k\ge p+1 $ satisfies that $ x_k\ge 0 $ , $ x_k\ne \max\limits_{j\in J_z}\{\sum\limits_{t=k}^{\psi(k)-1}d_{jt}-s_{j(k-1)}\} $, and for any $ x_q>0 $, $ p+1\le q<k $ there is $ x_q=\max\limits_{j\in J_z}\{\sum\limits_{t=q}^{\psi(q)-1}d_{jt}-s_{j(q-1)}\} $. 
	
	case 3.1. If $ x_k<\max\limits_{j\in J_z}\{\sum\limits_{t=k}^{\psi(k)-1}d_{jt}-s_{j(k-1)}\} $, then there exist at least one $ j\in J_z $ such that demand $ d_{j(\psi(k)-1)} $ is not satisfied.
	
	case 3.2. If $ x_k>\max\limits_{j\in J_z}\{\sum\limits_{t=k}^{\psi(k)-1}d_{jt}-s_{j(k-1)}\} $, let functions $ g_3 $ and $ f $ respectively be 
	\begin{equation*}
	g_3(x,y,s,s_j,z)=\sum\limits_{i=1}^{p} (\alpha_i x_i+\beta_i y_i+h_i s_i)+\sum\limits_{k\in[p+1,T]\setminus[k,\psi(k)]}(\alpha_i x_i+\beta_i y_i)+\sum\limits_{j\in{J_z}}(p_j \sum\limits_{i\in [p+1,T]\setminus[k,\psi(k)]} h_i s_{ji}) 				
	\end{equation*}
	\begin{equation*}
	\begin{aligned}
	f(x,y,s,s_j,z)&=g_3(x,y,s,s_j,z)+(\alpha_k x_k+\beta_k y_k)+\sum\limits_{j\in{J_z}}(p_j\sum\limits_{i = k}^{\psi(k)-1}h_is_{ji})\\
	&+(\alpha_{\psi(k)} x_{\psi(k)}+\beta_{\psi(k)} y_{\psi(k)}+\sum\limits_{j\in J_z}p_jh_{\psi(k)}s_{j\psi(k)}) 
	\end{aligned}
	\end{equation*}
	and let $ (\bar{x}, \bar{y}, \bar{s}, \bar{s}_j, \bar{z}) $ be \\
	$ \bar{x}_k=x_k-\epsilon $, $ \bar{s}_{ji}=s_{ji}-\epsilon $ for any $ i\in [k, \psi(k)-1], j\in J_z $ and $ \bar{x}_{\psi(k)}=x_{\psi(k)}+\epsilon $, and other components are the same as $ (x,y,s,s_j,z) $.
	
	Then we have 
	\begin{equation*}
	\begin{aligned}
	f(\bar{x},\bar{y},\bar{s},\bar{s}_j,z)&=g_3(x,y,s,s_j,z)+(\alpha_k (x_k-\epsilon)+\beta_k y_k)+\sum\limits_{j\in J_z}[p_j\sum\limits_{i = k}^{\psi(k)-1}h_i(s_{ji}-\epsilon)]\\
	&+(\alpha_{\psi(k)} (x_{\psi(k)}+\epsilon)+\beta_{\psi(k)} y_{\psi(k)}+\sum\limits_{j\in J_z}p_j h_{\psi(k)} s_{j\psi(k)})\\
	&=f(x,y,s,s_j,z)-(\alpha_k+\sum\limits_{j\in J_z}p_j\sum\limits_{i = k}^{\psi(k)-1}h_i-\alpha_{\psi(k)})\epsilon\\
	&=f(x,y,s,s_j,z)-[\sum\limits_{i = k}^{\psi(k)-1}(\alpha_i+\sum\limits_{j\in J_z}p_j h_i-\alpha_{i+1})]\epsilon
	\end{aligned}
	\end{equation*}
	by Assumption 1, we have $ \sum\limits_{i = k}^{\psi(k)-1}(\alpha_i+\sum\limits_{j\in J_z}p_j h_i-\alpha_{i+1})\ge \sum\limits_{i = k}^{\psi(k)-1}(\alpha_i+(1-\epsilon) h_i-\alpha_{i+1})\ge 0 $, if $ \sum\limits_{i = k}^{\psi(k)-1}(\alpha_i+\sum\limits_{j\in J_z}p_j h_i-\alpha_{i+1})>0 $, then $ f(\bar{x},\bar{y},\bar{s},\bar{s}_j,\bar{z})< f(x,y,s,s_j,z) $, for any $ 0<\epsilon\le x_k-\max\limits_{j\in J_z}\{\sum\limits_{t=k}^{\psi(k)-1}d_{jt}-s_{j(k-1)}\} $, which contradicts to the optimal property of $ (x,y,s,s_j,z) $. Therefore, $ \sum\limits_{i = k}^{\psi(k)-1}(\alpha_i+\sum\limits_{j\in J_z}p_j h_i-\alpha_{i+1})=0 $, then $ \epsilon $ can be increased such that $ \bar{x}_k=\max\limits_{j\in J_z}\{\sum\limits_{t=k}^{\psi(k)-1}d_{jt}-s_{j(k-1)}\} $ without affecting the optimal property.
	
	Now we have proved that there is one optimal $ x $ has the expression of equation (13), then we prove the rest equations.
	
	By constraints $ x_i+s_{i-1}=d_i+s_i, i=[1,p-1] $, it is easy to clarify that equation (14) holds.
	
	Apparently, for any $ j\in \Omega\setminus J_z $, $ z_j=1 $ and the optimal $ s_{ji} $, $ i\in [p+1,T] $, then we consider the optimal $ s_j $ when $ z_j=0 $.
	
	By (14), there is $ s_p=\max\limits_{j\in{J_z}}\{\sum\limits_{t=p+1}^{\psi(l)-1}d_{jt}\} $.
	
	Under optimal condition, the constraints about $ s_{ji}, j\in J_z $ should all reach equiality. When $ i\in [p+1, \psi(l)-1] $, $ y_i=0 $, so $ x_i=0 $, then 
	for each $ i\in[p+1,\psi(l)-1] $, $ s_{ji}=\sum\limits_{t=p+1}^{i}(x_t-d_{jt})+s_p=\max\limits_{j\in J_z}\{\sum\limits_{p+1}^{\psi(l)-1}d_{jt}\}-\sum\limits_{t=p+1}^i d_{jt} $, and $ s_{j\psi(l)}=\sum\limits_{t=p+1}^{\psi(l)}(x_t-d_{jt})+s_p=x_{\psi(l)}+\max\limits_{j\in{J_z}}\{\sum\limits_{t=p+1}^{\psi(l)-1}d_{jt}\}-\sum\limits_{p+1}^{\psi(l)}d_{jt}=x_{\psi(l)}+s_{j(\psi(l)-1)}-d_{j\psi(l)} $.
	
	Assume that $ \psi(l)=i_1<i_2<\ldots<i_v\le T<i_{v+1}=T $, $ \mathcal{I}=\{i_1,\cdots,i_v\} $, and $ y_i=1 $ when $ i\in \mathcal{I} $, $ y_i=0 $ when $ i\in[p+1,T]\setminus \mathcal{I} $. We complete the proof of equation (3) by induction.
	
	When $ k=1 $, clearly, $ s_{ji_1}=x_{i_1}+s_{j(i_1-1)}-d_{ji_1} $, and for each $ i\in [i_1+1,i_2-1] $, there is $ s_{ji}=\sum\limits_{t=p+1}^{i}(x_t-d_{jt})+s_p=x_{i_1}+\max\limits_{j\in{J_z}}\{\sum\limits_{t=p+1}^{i_1-1}d_{jt}\}-\sum\limits_{t=p+1}^i d_{jt}=x_{\phi(i)}+s_{j(\phi(i)-1)}-\sum\limits_{t=\phi(i)}^i d_{jt} $.
	
	Assume that for each $ k_0\le k $, 
	\begin{equation}
	s_{ji_{k_0}}=x_{i_{k_0}}+s_{j(i_{k_0}-1)}-d_{ji_{k_0}}
	\end{equation}
	and 
	\begin{equation}
	s_{ji}=\sum\limits_{t=p+1}^{i}(x_t-d_{jt})+s_p=x_{\phi(i)}+s_{j(\phi(i)-1)}-\sum\limits_{t=\phi(i)}^i d_{jt}, \quad i\in[i_{k_0}+1,i_{k_0+1}-1] 
	\end{equation}
	hold.
	
	Then 
	\begin{equation*}
	\begin{aligned}
	s_{ji_{k+1}}&=\sum\limits_{t=p+1}^{i_{k+1}}(x_t-d_{jt})+s_p\\
	&=(x_{i_{k+1}}-d_{ji_{k+1}})+\sum\limits_{t=p+1}^{i_{k+1}-1}(x_t-d_{jt})+s_p \\
	&=x_{i_{k+1}}-d_{ji_{k+1}}+s_{j(i_{k+1}-1)}\qquad(by\ \ Eq.(64))
	\end{aligned}
	\end{equation*}
	and for each $ i\in[i_{k+1}+1,i_{k+2}-1] $
	\begin{equation*}
	\begin{aligned}
	s_{ji}&=\sum\limits_{t=p+1}^{i}(x_t-d_{jt})+s_p\\
	&=\sum\limits_{t=p+1}^{i_{k+1}-1}(x_t-d_{jt})+s_p+\sum\limits_{t=i_{k+1}}^{i}(x_t-d_{jt})\\
	&=s_{j(i_{k+1}-1)}+x_{i_{k+1}}-\sum\limits_{t=i_{k+1}}^{i}d_{jt}\\
	&=x_{\phi(i)}+s_{j(\phi(i)-1)}-\sum\limits_{t=\phi(i)}^{i}d_{jt}
	\end{aligned}
	\end{equation*}
	
	Hence, (15) holds.

	\textbf{Proof of Proposition 2:} Let $  l=argmax\{i: y_i=1, i\le p\} $, then $ \psi(l)\ge p+1 $. 
	
	Assume that $ \psi(l)=i_1<i_2<\ldots<i_v\le T<i_{v+1}=T $, $ \mathcal{I}=\{i_1,\cdots,i_v\} $, and $ y_i=1 $ when $ i\in \mathcal{I} $, $ y_i=0 $ when $ i\in[p+1,T]\setminus \mathcal{I} $. 
	
	If for any $ i_k \in \mathcal{I} $, 
	\begin{equation}
	s_{j(i_k-1)}=\max\limits_{j\in{J_z}}\{\sum\limits_{t=p+1}^{i_k-1}d_{jt}\}-\sum\limits_{t=p+1}^{i_k-1}d_{jt}
	\end{equation}
	and
	\begin{equation}
	x_{i_k}=\max\limits_{j\in J_z} \{\sum\limits_{t=p+1}^{i_{k+1}-1}d_{jt}\}-\max\limits_{j\in J_z}\{\sum\limits_{t=p+1}^{i_k-1}d_{jt}\}
	\end{equation}
	hold, then for each $ i\in [i_k+1,i_{k+1}-1] $, 
	\begin{equation*}
	\begin{aligned}
	s_{ji}&=x_{\phi(i)}+s_{j(\phi(i)-1)}-\sum\limits_{t=\phi(i)}^i d_{jt}\\
	&=x_{i_k}+s_{j(i_k-1)}-\sum\limits_{t=i_k}^i d_{jt}\\
	&=\max\limits_{j\in J_z} \{\sum\limits_{t=p+1}^{i_{k+1}-1}d_{jt}\}-\max\limits_{j\in J_z}\{\sum\limits_{t=p+1}^{i_k-1}d_{jt}\}+\max\limits_{j\in{J_z}}\{\sum\limits_{t=p+1}^{i_k-1}d_{jt}\}-\sum\limits_{t=p+1}^{i_k-1}d_{jt}-\sum\limits_{t=i_k}^id_{jt}\\
	&=\max\limits_{j\in J_z} \{\sum\limits_{t=p+1}^{i_{k+1}-1}d_{jt}\}-\sum\limits_{t=p+1}^id_{jt}
	\end{aligned}
	\end{equation*}
	which means (18) holds.
	
	Thus, we only need to justify (65) (66).We consider to prove that by induction.
	
	For $ i_1 $, by (15) we have
	\begin{equation*}
	s_{j(i_1-1)}=\max\limits_{j\in J_z}\{\sum\limits_{t=p+1}^{i_1-1}d_{jt}\}-\sum\limits_{t=p+1}^{i_1-1}d_{jt}
	\end{equation*}
	then 
	\begin{equation*}
	\begin{aligned}
	x_{i_1}&=\max\limits_{j\in J_z}\{\sum\limits_{t=i_1}^{i_2-1}d_{jt}-s_{j(i_1-1)}\}\\
	&=\max\limits_{j\in J_z}\{\sum\limits_{t=i_1}^{i_2-1}d_{jt}-\max\limits_{j\in{J_z}}\{\sum\limits_{t=p+1}^{i_1-1}d_{jt}\}+\sum\limits_{t=p+1}^{i_1-1}d_{jt}\}\\
	&=\max\limits_{j\in{J_z}}\{\sum\limits_{t=p+1}^{i_2-1}d_{jt}\}-\max\limits_{j\in{J_z}}\{\sum\limits_{t=p+1}^{i_1-1}d_{jt}\}
	\end{aligned}
	\end{equation*}
	
	Now assume that for any $ k_0\le k $, equations (65) (66) hold, then for $ i_{k+1} $, by (15) we have
	\begin{equation*}
	\begin{aligned}
	s_{ji} &= 
	\left\{ \begin{array}{lll} 
	x_{i_k}+s_{j(i_k-1)}-d_{ji_k},\quad i_k=i_{k+1}-1 \\
	x_{\phi(i_{k+1}-1)}+s_{j(\phi(i_{k+1}-1)-1)}-\sum\limits_{t=\phi(i_{k+1}-1)}^{i_{k+1}-1}d_{jt}, \quad i_k<i_{k+1}-1
	\end{array} 
	\right. \\
	&= 
	\left\{ \begin{array}{lll} 
	x_{i_k}+s_{j(i_k-1)}-d_{ji_k},\quad i_k=i_{k+1}-1 \\
	x_{i_k}+s_{j(i_{k}-1)}-\sum\limits_{t=i_k}^{i_{k+1}-1}d_{jt}, \quad i_k<i_{k+1}-1
	\end{array} 
	\right. \\
	&=x_{i_k}+s_{j(i_k-1)}-\sum\limits_{t=i_k}^{i_{k+1}-1}d_{jt}\\
	&=\max\limits_{j\in J_z} \{\sum\limits_{t=p+1}^{i_{k+1}-1}d_{jt}\}-\max\limits_{j\in J_z}\{\sum\limits_{t=p+1}^{i_k-1}d_{jt}\}+\max\limits_{j\in{J_z}}\{\sum\limits_{t=p+1}^{i_k-1}d_{jt}\}-\sum\limits_{t=p+1}^{i_k-1}d_{jt}-\sum\limits_{t=i_k}^{i_{k+1}-1}d_{jt}\\
	&=\max\limits_{j\in J_z} \{\sum\limits_{t=p+1}^{i_{k+1}-1}d_{jt}\}-\sum\limits_{t=p+1}^{i_{k+1}-1}d_{jt}
	\end{aligned}
	\end{equation*}
	and
	\begin{equation*}
	\begin{aligned}
	x_{i_{k+1}}&=\max\limits_{j\in J_z}\{\sum\limits_{t=i_{k+1}}^{\psi(i_{k+1})-1}d_{jt}-s_{j(i_{k+1}-1)}\}\\
	&=\max\limits_{j\in J_z}\{\sum\limits_{t=i_{k+1}}^{i_{k+2}-1}d_{jt}-\max\limits_{j\in J_z}\{\sum\limits_{t=p+1}^{i_{k+1}-1}d_{jt}\}+\sum\limits_{t=p+1}^{i_{k+1}-1}d_{jt}\}\\
	&=\max\limits_{j\in J_z} \{\sum\limits_{t=p+1}^{i_{k+2}-1}d_{jt}\}-\max\limits_{j\in J_z}\{\sum\limits_{t=p+1}^{i_{k+1}-1}d_{jt}\}
	\end{aligned}
	\end{equation*}
	
	Hence, (65) (66) hold.
\end{APPENDIX}




\end{document}